\documentclass[reqno]{amsart}
\usepackage{amsthm, amsmath, amssymb,amscd}
\usepackage{xcolor,hyperref}
\usepackage{enumitem}
\newtheorem{theorem}{Theorem}[section]

\newtheorem{corollary}[theorem]{Corollary}
\newtheorem{proposition}[theorem]{Proposition}
\theoremstyle{definition}

\newtheorem{example}[theorem]{Example}

\theoremstyle{remark}
\newtheorem{remark}[theorem]{Remark}

\numberwithin{equation}{section}
\theoremstyle{definition}

\allowdisplaybreaks

\begin{document}
	
	\title[Discriminant and Index]{On the discriminant and index of a certain class of polynomials}

	\author{Rupam Barman}
	\address{Department of Mathematics, Indian Institute of Technology Guwahati, Assam, India, PIN-781039}
	\email{rupam@iitg.ac.in \\
		\href{https://orcid.org/0000-0002-4480-1788}{ORCID: 0000-0002-4480-1788}}
	
	\author{Anuj Narode}
	\address{Department of Mathematics, Indian Institute of Technology Guwahati, Assam, India, PIN- 781039}
	\email{anujanilrao@iitg.ac.in \\
		\href{https://orcid.org/0009-0005-6643-3401}{ORCID: 0009-0005-6643-3401}}
	
	\author{Vinay Wagh}
	\address{Department of Mathematics, Indian Institute of Technology Guwahati, Assam, India, PIN- 781039}
	\email{vinay\_wagh@yahoo.com \\
		\href{https://orcid.org/0000-0003-1977-464X}{ORCID: 0000-0003-1977-464X}}

	\date{\today}
	
	\thanks{}
	
	\subjclass[2010]{11R04; 11R09; 12F05}
	
	\keywords{Monogenity; power integral basis; discriminant; index of polynomial}
	
	\dedicatory{}
	\begin{abstract}
		 Let $f(x) = (x^{2}+1)^{n} - a x^{n} \in \mathbb{Z}[x]$ and assume $f(x)$ is irreducible. Let $\theta$ be a root of $f(x)$, set $K= \mathbb{Q}(\theta)$, and denote by $\mathbb{Z}_{K}$ the ring of integers of $K$. The index of $f$, denoted $\operatorname{ind}(f)$, is the index of $\mathbb{Z}[\theta]$ in $\mathbb{Z}_{K}$. A polynomial $f(x)$ is said to be monogenic if $\operatorname{ind}(f) = 1$.
		 In this article, we explicitly compute the discriminant of the polynomial $f(x)$, and then derive necessary and sufficient conditions on the parameters $a$ and $n$ for $f(x)$ to be monogenic. Furthermore, we provide a complete description of the primes that divide $\operatorname{ind}(f)$.
	\end{abstract}
	\maketitle
		\section{Introduction and Statement of Results}
		Let $f(x) \in \mathbb{Z}[x]$ be a monic irreducible polynomial with a root $\theta$ and $K = \mathbb{Q}(\theta)$. We denote by $\mathbb{Z}_K$ the ring of integers of $K$. It is well-known that $\mathbb{Z}_K$ is a free $\mathbb{Z}$-module of rank $\deg(f)$, and $\mathbb{Z}[\theta]$ is a submodule of $\mathbb{Z}_K$ of the same rank so that $[\mathbb{Z}_K : \mathbb{Z}[\theta]]$ is finite. The index of $f(x)$, denoted $\operatorname{ind}(f)$, is defined to be the index $[ \mathbb{Z}_K : \mathbb{Z}[\theta]]$.  
		The polynomial $f(x)$ is called monogenic if the index of $f(x)$ is one, i.e., $\mathbb{Z}_K = \mathbb{Z}[\theta]$. Further, the number field $K$ is said to be monogenic if there exists some $\alpha \in \mathbb{Z}_K$ such that $\mathbb{Z}_K = \mathbb{Z}[\alpha]$. Note that the monogenity of the polynomial $f(x)$ implies the monogenity of the number field $K$. However, the converse is false in general. For example, let $K = \mathbb{Q}(\theta)$, where $\theta$ is a root of $f(x)=x^2-5$. Then $K$ is monogenic, since $\mathbb{Z}_K = \mathbb{Z}\left[\frac{1 + \sqrt{5}}{2}\right]$, whereas $f(x)$ is not monogenic because $[\mathbb{Z}_K: \mathbb{Z}[\sqrt{5}]] = 2$.
		\par
		We let $\Delta(f)$ and $\Delta(K)$ denote the discriminants of the polynomial $f(x)$ and the number field $K$, respectively. It is well-known that $\Delta(f)$ and $\Delta(K)$ are related by the formula
		\begin{equation} \label{eq-2.1}
			\Delta(f) = \operatorname{ind}(f)^2 \cdot \Delta(K).	
		\end{equation}
		The determination of whether an algebraic number field is monogenic is a classical and fundamental problem in algebraic number theory. We easily see from \eqref{eq-2.1} that if $\Delta(f)$ is squarefree, then $f(x)$ is monogenic. However, the converse does not hold in general. In this connection, several classes of monogenic polynomials with non-squarefree discriminant have been investigated, see for example \cite{barman2025, Jones_2019, Jones-acta}. When the discriminant of $f(x)$ is not squarefree,
		it can be quite difficult to establish that $[\mathbb{Z}_K : \mathbb{Z}[\theta]]=1$. One of the standard techniques used is known as Dedekind's Index Criterion \cite[Theorem~6.1.4]{cohen}. This method is applied to determine whether or not a particular prime $p$ is a divisor of $[\mathbb{Z}_K : \mathbb{Z}[\theta]]$. Using this criterion, Jakhar~\emph{et al.} \cite{Khanduja} obtained necessary and sufficient conditions for determining the primes dividing the index of a trinomial. In \cite{Jones-2019}, Jones computed the discriminant of the polynomial $g(x)=x^{n}+a(bx+c)^{n}\in \mathbb{Z}[x]$ with $1\leq m<n$ and proved that when $\gcd(n,\,mb)=c=1$, there exist infinitely many values of $a$ for which $g(x)$ is irreducible and monogenic. In the same article, he conjectured that if $\gcd(n,\,mb)=1$ 
		and $a$ is prime, then $g(x)$ is monogenic if and only if 
		$n^{n}+(-1)^{\,n+m} b^{n} (n-m)^{\,n-m} m^{m}a$
		is squarefree.
		This conjecture was subsequently proved by Kaur~\emph{et al.} in \cite{Surender-2024}. 
		In \cite{Jakhar-2024}, Jakhar established necessary and sufficient conditions for the primes dividing 
		the index of $g(x)$, and further proved that, if $\deg(g)=q$ is prime and if there exists a prime 
		$p$ such that $p\mid \Delta(g)$, $p^{2}\nmid \Delta(g)$, and $p\nmid abcm$, then the Galois group of 
		$g(x)$ is isomorphic to the symmetric group $S_{q}$.
		Later, in \cite{Ravi}, Jakhar~\emph{et al.} considered the more general polynomial
		$g_1(x)=x^{n}+a(bx^{k}+c)^{n}$, and studied both its discriminant and the primes dividing its index. In a similar vein, Jones~\cite{Jones-acta} determined the discriminant of the family  $g_{a}(x)=x^{\,n-m}(x+k)^{m}+a$
		and proved that there exist infinitely many primes $p$ for which $g_{p}(x)$ is monogenic. More 
		recently, Jakhar~\emph{et al.} \cite{Prabhakr-2025} characterized the primes dividing the index of 
		$g_{a}(x)$.
		\par 
		In this article, we study the discriminant and index of the polynomial $f(x) = (x^2 + 1)^n - ax^n$, where $a\in \mathbb{Z}$. In the following theorem, we calculate the discriminant of $f(x)$.
		\begin{theorem} \label{thm-1.1}
			Let $f(x) = (x^2 + 1)^n - a x^n \in \mathbb{Z}[x]$, where $n \ge 2$. If $f(x)$ is irreducible, then  
			\begin{equation}\label{eq-1}
				\Delta(f) = (-1)^{2n \choose 2} n^{2n} a^{2n-2}(2^n - a)(2^n - (-1)^n a).
			\end{equation}	
		\end{theorem}
		In the following theorem, we characterize all the primes dividing the index of $f(x)$. 
		\begin{theorem} \label{Thm-1.2}
			Let $K = \mathbb{Q}(\theta)$ be an algebraic number field with $\theta$ in the ring $\mathbb{Z}_K$ of algebraic integers of $K$ having minimal polynomial $f(x)  = (x^2 + 1)^n - ax^n$ over $\mathbb{Q}$. A prime factor $p$ of the discriminant $\Delta(f)$ of $f(x)$ does not divide $[\mathbb{Z}_K : \mathbb{Z}[\theta]]$ if and only if $p$ satisfies one of the following conditions:
			\begin{enumerate}[label=\textup{(\roman*)}]
				\item when $p \mid a$, then $p^2 \nmid a$;
				\item when $p \nmid a$ and $p \mid n$, then $p^2 \nmid (a^{p^j} - a)$ with $j$ as the highest power of $p$ dividing $n$;
				\item when $p \nmid an$ and $n$ is odd, then $p^2$ does not divide $\Delta(f)$;
				\item  when $p \nmid an$ and $n$ is even, then $p^2$ does not divide $2^n - a$.
			\end{enumerate}
		\end{theorem}
		The following corollary is an immediate consequence of Theorem~\ref{Thm-1.2}.
		\begin{corollary}
			Let $K = \mathbb{Q}(\theta)$ be an algebraic number field with $\theta$ in the ring $\mathbb{Z}_K$ of algebraic integers of $K$ having minimal polynomial $f(x)  = (x^2 + 1)^n - ax^n$ over $\mathbb{Q}$. Then $\mathbb{Z}_K =~\mathbb{Z}[\theta]$ if and only if each prime $p$ dividing $\Delta(f)$ satisfies one of the conditions (i)-(iv) of Theorem~\ref{Thm-1.2}.
		\end{corollary}
		\section{Proof of Theorem \ref{thm-1.1} and \ref{Thm-1.2}}
		 Dedekind's criterion \cite[Theorem~6.1.4]{cohen} has been extensively studied and generalized in the literature (see, for example, \cite{Ershov, Khanduja-2016, Khanduja-2010}). In \cite{Khanduja-2016}, Khanduja \emph{et al.} studied the equivalent versions of the generalized Dedekind's criterion. We use the following equivalent version to prove Theorem~\ref{Thm-1.2}. Also, see Lemma~2.1 in \cite{Jakhar-JNT}.
		\begin{theorem} \cite[Theorem 1.1]{Khanduja-2016} \label{Thm-2.1}
			Let $f(x) \in \mathbb{Z}[x]$ be a monic irreducible polynomial having the factorization $\overline{g}_1(x)^{e_1} \cdots \overline{g}_t(x)^{e_t}$ modulo a prime $p$ as a product of powers of distinct irreducible polynomials over $\mathbb{Z}/p\mathbb{Z}$ with each $g_i(x) \in \mathbb{Z}[x] $ monic. Let $K = \mathbb{Q}(\theta)$ with $\theta$ a root of $f(x)$. Then the following are equivalent:
			\begin{enumerate}[label=\textup{(\roman*)}]
				\item $p$ does not divide $[\mathbb{Z}_K : \mathbb{Z}[\theta]]$;
				\item for each $i$, either $e_i = 1$ or $\overline{g}_i(x)$ does not divide $\overline{M}(x)$ where, 
							\begin{equation*}
								M(x) = \frac{1}{p} \left(f(x) - g_1(x)^{e_1} \cdots g_t(x)^{e_t}\right);
							\end{equation*}
				\item $f(x)$ does not belong to the ideal $\langle p, g_i(x) \rangle^2$ in $\mathbb{Z}[x]$ for any $i$, $1 \leq i \leq t.$
			\end{enumerate}
		\end{theorem}
		Next, we recall a proposition which will be used to calculate $\Delta(f)$ in Theorem~\ref{thm-1.1}. 
		\begin{proposition} \cite[Proposition 12.1.4]{Ireland-1990} \label{Thm-2.3}
			Let $f(x) \in \mathbb{Z}[x]$ be a monic irreducible polynomial of degree $n$. Let $\theta$ be a root of $f(x)$ and $K = \mathbb{Q}(\theta)$. Then $$ \Delta(f) = (-1)^{n \choose 2} \mathcal{N}(f'(\theta)),$$ where $\mathcal{N} : = \mathcal{N}_{K / \mathbb{Q}}$ is the algebraic norm.
		\end{proposition}
		 Now we prove our main results. We first prove Theorem~\ref{thm-1.1}. 
		\begin{proof}[Proof of Theorem~\ref{thm-1.1}]
			Suppose $f(\theta) = 0$. Then,
				\begin{equation}
					(\theta^2  + 1)^n = a \theta^n.  \label{eq-1.2}
				\end{equation}
			Since $f'(x) = 2nx(x^2 + 1)^{n-1} - anx^{n-1}$, \eqref{eq-1.2} yields
				\begin{align*}
					(\theta^2 + 1) f'(\theta) & = 2n\theta(\theta^2 + 1)^{n} - an\theta^{n-1}(\theta^2 + 1) \\
											 & = 2n \theta a \theta^n -na\theta^{n-1}(\theta^2 + 1) \\
											 & = na\theta^{n-1}(\theta -1 )(\theta + 1). 
				\end{align*}
			Hence,
				  \begin{equation}
				  	\mathcal{N}(f'(\theta)) = \frac{\mathcal{N}(a)\mathcal{N}(n)\mathcal{N}(\theta)^{n-1}\mathcal{N}(\theta-1)\mathcal{N}(\theta+1)}{\mathcal{N}(\theta^2 + 1)}. \label{eq-1.3}
				\end{equation}
			Note that if $f(x)$ is the minimal polynomial of $\theta$, then $f(x - 1)$ and $f(x + 1)$ are the minimal polynomials of $\theta + 1$ and $\theta - 1$, respectively. Thus, $\mathcal{N}(\theta) = 1$, $\mathcal{N}(\theta + 1) = (2^n - (-1)^na)$ and $\mathcal{N}(\theta - 1) = (2^n - a)$. Further, from \eqref{eq-1.2}, we have $\mathcal{N}(\theta^2 + 1)^n = \mathcal{N}(a)\mathcal{N}(\theta)^n = a^{2n}  $, so that $\mathcal{N}(\theta^2 + 1) = a^2$. Now, \eqref{eq-1.3} yields
				\begin{align*}
					\mathcal{N}(f'(\theta)) & = \frac{a^{2n}n^{2n}(2^n -a )(2^n -(-1)^na)}{a^2} \\ 
											& = a^{2n-2}n^{2n}(2^n -a )(2^n -(-1)^na).
				\end{align*}
			Hence, employing Proposition~\ref{Thm-2.3} we have
			\begin{equation*}
				 \Delta(f) = (-1)^{2n \choose 2} \mathcal{N}(f'(\theta))= (-1)^{2n \choose 2} n^{2n} a^{2n-2}(2^n - a)(2^n - (-1)^n a).
			\end{equation*} 
			This completes the proof of the theorem.
		\end{proof}
	Next, we prove Theorem~\ref{Thm-1.2}.
	\begin{proof}[Proof of Theorem~\ref{Thm-1.2}] Let $p$ be a prime dividing $\Delta(f)$. In view of Theorem~\ref{Thm-2.1}, $p$ does not divide $[\mathbb{Z}_K : \mathbb{Z}[\theta]]$ if and only if $f(x)\notin\langle p, g(x) \rangle^2$ for any monic polynomial $g(x)\in \mathbb{Z}[x]$ which is irreducible modulo $p$. We prove the theorem by considering the following cases.\\\\
	\underline{Case (i)}: $p \mid a$. Then $f(x) \equiv (x^2 + 1)^n \pmod p$. If $\overline{g}(x)$ is an irreducible factor of $x^2 + 1$ over $\mathbb{Z}/p\mathbb{Z}$, then $f(x) \not \in \langle p, g(x) \rangle^2$ if and only if $p^2 \nmid a$. Thus, by Theorem~\ref{Thm-2.1},  $p \nmid [\mathbb{Z}_K : \mathbb{Z}[\theta]]$ if and only if $p^2 $ does not divide $a$.\\\\
	\underline{Case (ii)}: $p \mid n $ and $p \nmid a$. Let $n = sp^j$ with $\gcd(s,p) = 1$. By  the Binomial theorem, $$f(x)  \equiv ((x^2+1)^s - ax^s)^{p^j}  \hspace{-10pt}\pmod p.$$
	Let $h(x) = (x^2+1)^s - ax^s$ and $\prod_{i=1}^{t} \overline{g_i}(x)^{e_i}$ be the factorization $h(x)$ over $\mathbb{Z}/ p \mathbb{Z}$, where $g_i(x)\in \mathbb{Z}[x]$ are monic polynomials which are distinct and irreducible modulo $p$. Now, raising both sides of $(x^2+ 1)^s  =	h(x) + ax^s$ to the $p^j$th power, we obtain
	$$(x^2 + 1)^n  = h(x)^{p^j} + a^{p^j} x^{s{p^j}} + p h(x) T(x),$$ for some polynomial $T(x)\in \mathbb{Z}[x]$, where $n=sp^j$. This yields
	\begin{equation}\label{eq-3.5}
		f(x)   = 	h(x)^{p^j} +  p h(x) T(x)  +(a^{p^j} -a) x^{n}.
 	\end{equation}
	 Therefore, by \eqref{eq-3.5}, $f(x) \not \in  \langle p , g_i(x) \rangle^2$ if and only if $p^2$ does not divide $a^{p^j} -a$. Thus, by Theorem~\ref{Thm-2.1}, $p \nmid [\mathbb{Z}_K : \mathbb{Z}[\theta]]$ if and only if $p^2$ does not divide $a^{p^j} -a$. This completes the proof in this case.\\\\
	\underline{Case (iii)}: $p \nmid an$ and $n$ is odd. Since $p \mid \Delta(f)$, the discriminant of $\overline{f}(x) $ is $0$ in $\mathbb{Z}/ p\mathbb{Z}$  so that $\overline{f}(x)$ has a repeated root in the algebraic closure of $\mathbb{Z}/p\mathbb{Z}$. Let $\alpha$ be a repeated root of $\overline{f}(x)$. Then,
		\begin{equation}
			\overline{f}(\alpha) = (\alpha^2 + 1)^n - \overline{a}\alpha^n = 0, \label{eq-3.2}
		\end{equation}		
		\begin{equation}
			 \overline{f}'(\alpha) = \overline{2n} \alpha (\alpha^2 + 1 )^{n-1} - \overline{n} \overline{a} \alpha^{n-1} = 0. \label{eq-3.3}
		\end{equation}
		Substituting  $(\alpha^2 + 1)^n = \overline{a}\alpha^n$ in \eqref{eq-3.3} yields $\overline{an} \alpha^{n-1} (\alpha^2 - 1) = 0$.
	Since $p \nmid an$ and $\alpha \neq 0$, any repeated root must satisfy 
	$\alpha^2 = 1$ in the algebraic closure of $\mathbb{Z}/p\mathbb{Z}$. If $\alpha$ is an integer such that $\alpha \equiv \pm 1 \pmod{p}$, then by \eqref{eq-3.2} we obtain $a \equiv 2^n \alpha^{-n} \pmod{p}$. In particular, we have 
	\begin{align}\label{Eq-2.6}
	a\equiv 
	\left\{\begin{array}{ll}
	\hspace{.3cm} 2^n \hspace{.3cm}\pmod{p}, & \hbox{if $\alpha\equiv 1 \hspace{.3cm}\pmod{p}$;} \\
	(-2)^n \pmod{p}, & \hbox{if $\alpha\equiv -1\pmod{p}$.}
	\end{array}
	\right.
	\end{align}	
	Both the congruences in \eqref{Eq-2.6} can hold simultaneously only if $p = 2$. However, this would force $2 \mid a$, contradicting the assumption that $p \nmid a$. Hence, either $\alpha \equiv 1 \pmod{p}$ or $\alpha \equiv -1 \pmod{p}$ is a root of $\overline{f}(x)$, but not both.
	\par 
	Next, we show that any multiple root of $\overline{f}(x)$ has multiplicity two, 
	i.e., $\overline{f}''(\alpha) \neq 0$. We have
		\begin{equation}
			f''(x) = n \left(4x^2 (n-1)(x^2 + 1)^{n-2} + 2(x^2 +1 )^{n-1} -a (n-1)x^{n-2}\right). \label{eq-3.4}
		\end{equation}
		Substituting $\alpha^2 = 1$ and $a\alpha^n \equiv 2^n\pmod{p}$ in \eqref{eq-3.4} yield $\overline{f}''(\alpha ) = \overline{n2^{n}}$. As $p \nmid n$, $f''( \alpha) \neq  0$ in the algebraic closure of $\mathbb{Z}/p\mathbb{Z}$. Hence,  $\alpha$ has multiplicity two.
		\par Now, let $\beta$ be an integer satisfying $\beta \equiv -1 \pmod {p^2}$ and $\alpha \equiv \beta \pmod p$. Then, we have
		\begin{align}
		f(x) \nonumber &= \left( (x - \beta + \beta)^2 + 1\right)^n - a\left(x -\beta + \beta\right)^n \\ \nonumber
		&= \left((x -\beta)^2 + 2\beta(x -\beta) + \beta^2 + 1\right)^n - a\left( (x-\beta)^2 + \beta \right)^n \\ \nonumber
		&=  (x-\beta)^{2n} + {n\choose 1}(x -\beta)^{2n-1}(2\beta(x -\beta) + \beta^2 + 1) + \cdots  \\ \nonumber
		&+  {n \choose {n-1}} (x - \beta)^{2n-n+1} (2\beta(x -\beta) + \beta^2 + 1)^{n-1} + (2\beta(x -\beta) + \beta^2 + 1)^n \\ \nonumber
		 &- a\left((x - \beta)^n + {n \choose 1} \beta (x -\beta)^{n-1} + \cdots + {n \choose 1} \beta^{n-1} (x - \beta)  + \beta^n \right) \\ \nonumber
		&=  (x - \beta)^2 g(x) + \left(2n\beta(\beta^2 + 1)^{n-1} - na \beta^{n-1}\right)(x - \beta) + \left((\beta^2 +1)^n - a \beta^n\right) \\ \label{Eq-2.9}
		&=  (x - \beta)^2 g(x) + f'(\beta) (x - \beta) + f(\beta),
		\end{align}
		where 
		\begin{equation*}
		g(x) = \sum_{k = 1 }^{n} {n \choose k} u^{2k-2}(2 \beta u + t)^{n-k} - \sum_{k =2}^{n} a {n \choose k} u^{k-2} \beta^{n-k}
		\end{equation*} with $u = x -\beta $ and $t = \beta^2 + 1$. Alternatively, one can find the Taylor series expansion of $f(x)$ at $\beta$ as follows:
		\begin{align*}
		f(x)  &=   f(\beta) + (x - \beta) f'(\beta) + (x - \beta)^2 \frac{f''(\beta)	}{2} + \cdots + (x - \beta)^{2n} \frac{f^{2n}(\beta)}{(2n)!} \\ 
		& =   f(\beta) + (x - \beta) f'(\beta) + (x - \beta)^2 \left( f''(\beta)  + \cdots +   (x - \beta)^{2n -2} \frac{f^{2n}(\beta)}{(2n)!} \right).
		\end{align*}
		Then, we have 
		\[ \overline{f}(x) = (x - \beta)^2 \overline{g}(x).\]
		Note that $\overline{g}(x) \in (\mathbb{Z}/p\mathbb{Z})[x]$ is a separable polynomial, 
		since $\beta$ is the only repeated root of $\overline{f}(x)$ in the algebraic closure of $\mathbb{Z}/p\mathbb{Z}$.
		Hence, we have 
		\[ g(x) = \prod_{i=1}^{t} g_i(x) + p h(x),\]
		where $g_1(x), g_2(x), \ldots, g_t(x)$ are monic polynomials over $\mathbb{Z}$ which are distinct and irreducible modulo $p$, 
		and $h(x) \in \mathbb{Z}[x]$.
		Therefore, \eqref{Eq-2.9} yields
		\[ f(x) = (x - \beta)^2 \left(\prod_{i=1}^{t} g_i(x) + p h(x)\right) 
		+ (x - \beta) f'(\beta) + f(\beta).\]
		From Theorem \ref{Thm-2.1}, it follows that 
		\[ p \nmid [\mathbb{Z}_K : \mathbb{Z}[\theta]] \]
		if and only if $\gcd((x - \beta), \overline{M}(x)) = 1$, where
		\[ M(x) = \frac{1}{p}\left(f(x) - \prod_{i=1}^{t} g_i(x)\right)
		= \frac{1}{p}\left(p(x - \beta)^2 h(x) + (x - \beta) f'(\beta) + f(\beta)\right).\]
		Moreover, $\gcd((x - \beta), \overline{M}(x)) = 1$ 
		if and only if $f(\beta) \not\equiv 0 \pmod{p^2}$. Since $\beta \equiv -1 \pmod{p}$ we have $p \mid (2^n+a)$ and $p \nmid (2^n -a)$.
		In addition, $p \nmid an$. Consequently, $f(\beta) \not \equiv 0 \pmod {p^2}$ if and only if $p^2 \nmid( 2^n + a)$. As $p \nmid (2^n -a)$ and $p \nmid an$, it follows from \eqref{eq-1} that $$p \nmid [\mathbb{Z}_K : \mathbb{Z}[\theta]] \text{~if ~and ~only ~if~} p^2 \nmid \Delta(f).$$
		Proceeding similarly to the case $ \alpha \equiv -1 \pmod p$ as shown above, we can prove that, if $ \alpha \equiv 1 \pmod p$ then  $p \nmid [\mathbb{Z}_K : \mathbb{Z}[\theta]]$ if and only if $p^2 \nmid \Delta(f)$. This completes the proof of Case~(iii).
	\\\\ 
	\underline{Case (iv)}: $p \nmid an$ and $n$ is even. Since $n$ is even, from \eqref{Eq-2.6} we have $p \mid 2^n -a$ if $ \alpha \equiv \pm 1 \pmod p$. The rest of the proof goes along similar lines as shown in Case~(iii). This completes the proof of the theorem.
	\end{proof}
	\section{Examples and Remarks}
		\begin{example} \label{Ex-3.1}
		Let $p$ be an odd prime, and consider the polynomial $$
		f_p(x) = (x^{2}+1)^{p} - p x^{p}. $$
		Assume that $f_p(x)$ is irreducible over $\mathbb{Q}$. Using \texttt{SageMath}, we find that $f_p(x)$ is irreducible over $\mathbb{Q}$ for all primes $p<200$. By Theorem~\ref{thm-1.1}, we have $\Delta(f_p) = p^{4p-2} H(p)$, where $H(p) = (2^{p} - p)(2^{p} + p)$.
		Observe that for any odd prime $p$, we have $p \nmid H(p)$. Let $r$ be a prime divisor of $\Delta(f_p)$. If $r = p$, then by Theorem~\ref{Thm-1.2}(i) it follows that
		$$r \nmid \operatorname{ind}(f_p).$$
		If $r \neq p$, then $r \mid H(p)$, and by Theorem~\ref{Thm-1.2}(iii) we have $$
		\operatorname{ind}(f_p) = 1  \text{ if and only if } r^{2} \nmid H(p).$$
		Consequently, if $H(p)$ is squarefree, then $\operatorname{ind}(f_p)=1$, and hence the polynomial $f_p(x)$ is monogenic. We find that the set of primes up to
		$100$ for which $H(p)$ is squarefree is given by $\{3, 11, 13, 17, 19, 29, 37,47,67,71,73,89\}$. 
		\par We now compute the index of $f_p(x)$ for the remaining primes. For $p = 5$, the discriminant is
		\[ \Delta(f_5) = 5^{10} \cdot 3^{3} \cdot 37.\]
		By Theorem~\ref{Thm-1.2}{(i)}, $5$ does not divide
		$\operatorname{ind}(f_5)$, and by Theorem~\ref{Thm-1.2}{(iii)}, $37$ also does not divide $\operatorname{ind}(f_5)$.
		Moreover, applying Theorem~\ref{Thm-1.2}{(iii)} again shows that $3$ divides $\operatorname{ind}(f_5)$. Recalling the identity
		\[ \Delta(f_5) = \operatorname{ind}(f_5)^{2}\,\Delta(K) \]
		from \eqref{eq-2.1}, we conclude that $\operatorname{ind}(f_5)=3$. Applying the same reasoning, one can compute $\operatorname{ind}(f_p)$ for all the remaining primes $p<100$. In Table~\ref{Table1}, we summarize these index values.
		\begin{table}[h!]
			\caption{Values of $p$ and $\operatorname{ind}(f_p)$}\label{Table1}
			\centering
			\begin{minipage}{0.45\textwidth}
				\centering
				\begin{tabular}{cccc}
					\hline
					$p$ & $\operatorname{ind}(f_p)$&$p$ & $\operatorname{ind}(f_p)$\\
					\hline
					5 & 3 &47 & 5  \\
					7 & 33 &59 & 3 \\
					23 & 3 &61 & $21$ \\
					31 & 11 &79 & 3 \\
					41 & 3 &83 & 5 \\
					43 & 3 &97 & $15$ \\
					\hline
				\end{tabular}
			\end{minipage}
		\end{table}
	\end{example}
	\begin{remark}
		Numerical experiments conducted in \texttt{SageMath} suggest that there are many primes $p$
		for which $H(p)$ is squarefree. It would be interesting to establish, using analytic methods, the existence of infinitely many such primes.
	\end{remark}
	\begin{remark}
		Let $q\neq p$ be a prime dividing $\Delta(f_p)$. In Example~\ref{Ex-3.1}, we have $\nu_q(\Delta(f_p)) \le 3$. Consequently, Theorem~\ref{Thm-1.2}, together with \eqref{eq-1.2}, allows one to compute $\operatorname{ind}(f_p)$. However, when $\nu_q(\Delta(f_p)) \ge 4$, Theorem~\ref{Thm-1.2} alone is insufficient to decide whether $q^{2}$ divides $\operatorname{ind}(f_p)$.
	\end{remark}

\end{document}